\documentclass{article}
\usepackage{amsmath,amsthm,amssymb}
\usepackage{enumerate}
\usepackage{graphicx}
\usepackage{cite}
\usepackage{comment}
\usepackage{oands}
\usepackage{tabularx}
\usepackage{multirow}
\usepackage[margin=1in]{geometry}
\usepackage[pagewise,mathlines]{lineno}  
\usepackage[normalem]{ulem} 
\usepackage{subfigure}  
\usepackage[all]{xy}
\usepackage[usenames,dvipsnames]{xcolor}

\usepackage{array}

\usepackage[pdftitle={Dimension transformation formula for conformal maps into the complement of an SLE curve},
  pdfauthor={Ewain Gwynne, Nina Holden, and Jason Miller},
colorlinks=true,linkcolor=NavyBlue,urlcolor=RoyalBlue,citecolor=PineGreen,bookmarks=true,bookmarksopen=true,bookmarksopenlevel=2,unicode=true,linktocpage]{hyperref}

\numberwithin{equation}{section}

\theoremstyle{plain}
\newtheorem{thm}{Theorem}[section]

\newtheorem{lem}[thm]{Lemma}

\theoremstyle{definition}
\newtheorem{defn}[thm]{Definition}
\newtheorem{remark}[thm]{Remark}

\def\@rst #1 #2other{#1}
\newcommand\MR[1]{\relax\ifhmode\unskip\spacefactor3000 \space\fi
  \MRhref{\expandafter\@rst #1 other}{#1}}
\newcommand{\MRhref}[2]{\href{http://www.ams.org/mathscinet-getitem?mr=#1}{MR#2}}

\newcommand{\dsb}{\begin{adjustwidth}{2.5em}{0pt}
\begin{footnotesize}}
\newcommand{\dse}{\end{footnotesize}
\end{adjustwidth}}

\newcommand{\ssb}{\begin{adjustwidth}{2.5em}{0pt}}
\newcommand{\sse}{\end{adjustwidth}}

\newcommand{\aryb}{\begin{eqnarray*}}
\newcommand{\arye}{\end{eqnarray*}}
\def\alb#1\ale{\begin{align*}#1\end{align*}}
\newcommand{\eqb}{\begin{equation}}
\newcommand{\eqe}{\end{equation}}
\newcommand{\eqbn}{\begin{equation*}}
\newcommand{\eqen}{\end{equation*}}

\newcommand{\BB}{\mathbb}
\newcommand{\ol}{\overline}
\newcommand{\ul}{\underline}
\newcommand{\op}{\operatorname}

\newcommand{\eqD}{\overset{d}{=}}
\newcommand{\ep}{\epsilon}
\newcommand{\rta}{\rightarrow}

\newcommand{\wt}{\widetilde}
\newcommand{\wh}{\widehat} 
\newcommand{\mcl}{\mathcal}

\newcommand{\bdy}{\partial}

\newcommand*\patchAmsMathEnvironmentForLineno[1]{  \expandafter\let\csname old#1\expandafter\endcsname\csname #1\endcsname
  \expandafter\let\csname oldend#1\expandafter\endcsname\csname end#1\endcsname
  \renewenvironment{#1}     {\linenomath\csname old#1\endcsname}     {\csname oldend#1\endcsname\endlinenomath}}\newcommand*\patchBothAmsMathEnvironmentsForLineno[1]{  \patchAmsMathEnvironmentForLineno{#1}  \patchAmsMathEnvironmentForLineno{#1*}}\AtBeginDocument{\patchBothAmsMathEnvironmentsForLineno{equation}\patchBothAmsMathEnvironmentsForLineno{align}\patchBothAmsMathEnvironmentsForLineno{flalign}\patchBothAmsMathEnvironmentsForLineno{alignat}\patchBothAmsMathEnvironmentsForLineno{gather}\patchBothAmsMathEnvironmentsForLineno{multline}}

\newcommand{\R}{{\mathbb R}}

\begin{document}

\author{
\begin{tabular}{c}Ewain Gwynne\\[-5pt]\small MIT\end{tabular}\;
\begin{tabular}{c}Nina Holden\\[-5pt]\small MIT\end{tabular}\;
\begin{tabular}{c}Jason Miller\\[-5pt]\small Cambridge\end{tabular}}

\title{Dimension transformation formula for conformal maps into the complement of an SLE curve}
\date{   }

\maketitle 

\begin{abstract} 
We prove a formula relating the Hausdorff dimension of a deterministic Borel subset of $\mathbb R$ and the Hausdorff dimension of its image under a conformal map from the upper half-plane to a complementary connected component of an SLE$_\kappa$ curve for $\kappa \not =4$. Our proof is based on the relationship between SLE and Liouville quantum gravity together with the one-dimensional KPZ formula of Rhodes-Vargas (2011) and the KPZ formula of Gwynne-Holden-Miller (2015). 
As an intermediate step we prove a KPZ formula which relates the Euclidean dimension of a subset of an SLE$_\kappa$ curve for $\kappa \in (0,4)\cup(4,8)$ and the dimension of the same set with respect to the $\gamma$-quantum natural parameterization of the curve induced by an independent Gaussian free field, $\gamma = \sqrt \kappa \wedge (4/\sqrt\kappa)$.
\end{abstract}



\section{Introduction}
\label{sec-intro}

\subsection{Overview} 
\label{sec-overview}
The Schramm-Loewner evolution (SLE$_\kappa$) is a family of conformally invariant random fractal curves in two dimensions, originally introduced in~\cite{schramm0}. SLE curves arise as the scaling limit of a variety of models in statistical physics; see, e.g., \cite{lsw-lerw-ust, smirnov-ising, ss-explorer, ss-dgff, gl-contours}.

There are various ways to quantify the precise manner in which the SLE$_\kappa$ curve is fractal. Suppose, for concreteness, that $\eta$ is a chordal SLE$_\kappa$ from $0$ to $\infty$ in $\BB H$ with hulls $(K_t)$ and for $t>0$ let $f_t : \BB H\setminus K_t \rta \BB H$ be its time $t$ centered Loewner map (i.e.\ $f_t(\eta(t)) = 0$, $f_t(\infty) = \infty$, and $\lim_{z\rta\infty} f_t(z)/z =1$). The most basic measure of the fractality of $\eta$ is its Hausdorff dimension, which was shown to be $(1+\kappa/8) \wedge 2$ in~\cite{beffara-dim}. One can also consider the H\"older regularity of either $\eta$ itself or the maps $f_t^{-1}$. The optimal H\"older exponent for the SLE curve is computed in~\cite{lawler-viklund-holder} (see also~\cite{schramm-sle,lind-holder} for intermediate results). The optimal H\"older exponent for $f_t^{-1}$ is not known, but is estimated in~\cite{lind-holder}. 

Another way to measure the fractality of $\eta$ is the multifractal spectrum. This is the function which gives, for each $s\in [-1,1]$, the Hausdorff dimension of the set of points in $x \in \BB R$ such that $|(f_t^{-1})'(x + iy)| = y^{-s + o_y(1)}$ as $y \rta 0$.   
The multifractal spectrum was predicted non-rigorously by Duplantier in~\cite{dup-mf-spec-bulk} (see also~\cite{dup-hm99,dup-mf99} for earlier predictions in special cases) and computed rigorously in~\cite{gms-mf-spec}. 
There are a number of other quantities related to the multifractal spectrum which have been computed either rigorously or non-rigorously. These include the winding spectrum (predicted in~\cite{binder-dup-winding1,binder-dup-winding2}), higher multifractal spectra depending on the derivative behavior on both sides of the curve (predicted in~\cite{dup-higher-mf}), the integral means spectrum (rigorous computations of different versions given in~\cite{bel-smirnov-hm-sle,gms-mf-spec,dup-coefficient,ly-spec-avg,ly-spec-new}), the multifractal spectrum at the tip (computed in~\cite{lawler-viklund-tip}), and the boundary multifractal spectrum (computed in~\cite{abv-bdy-spec}). 

In this paper, we will consider a different measure of the fractality of $\eta$, namely the manner in which the Hausdorff dimension of a subset $Y\subset \BB R$ transforms under the inverse centered Loewner map $f_t^{-1}$ when $f_t^{-1}(Y) \subset \eta$. We will prove a formula (Theorem~\ref{thm-uniform-dim-time}) for $\dim_{\mcl H} f_t^{-1}(Y)$ in terms of $\dim_{\mcl H} Y$ when $Y$ is chosen independently of $\eta$. We also prove a variant (Theorem~\ref{thm-uniform-dim}) for a conformal map from $\BB H$ to a complementary connected component of a whole SLE$_\kappa$ curve, $\kappa \in (0,4) \cup (4,8)$. Our formula appears to be closely related to the multifractal spectrum of SLE; see Problem~\ref{item-multifractal} in Section~\ref{sec-open-problems}. 

Although our main results are statements about the conformal geometry of SLE, our proofs are based on the relationship between SLE and Liouville quantum gravity (LQG). For $\gamma \in (0,2)$ and a domain $D\subset \BB C$, Liouville quantum gravity is the random Riemannian metric
\eqb \label{eqn-lqg0}
e^{\gamma h}(dx^2 + d y^2),
\eqe 
where $h$ is some variant of the Gaussian free field~\cite{shef-gff,ss-contour} (GFF) on $D$ and $dx^2 + dy^2$ is the Euclidean metric. The formula~\eqref{eqn-lqg0} does not make rigorous sense since $h$ is only a distribution (in the sense of Schwartz), not a pointwise-defined function. However, one can make rigorous sense of the volume form associated with LQG. In particular, there exists a measure $\mu_h$ on $D$ which is the a.s.\ limit of regularized versions of $e^{\gamma h(z)} \,dz$, where $dz$ is the Euclidean volume form (i.e., Lebesgue measure). This measure is a special case of \emph{Gaussian multiplicative chaos}~\cite{kahane,rhodes-vargas-review}; see also~\cite{shef-kpz}.
The measure $\mu_h$ is called the \emph{$\gamma$-quantum area measure} induced by $h$. Similarly, there is also a \emph{$\gamma$-quantum length measure} $\nu_h$ defined on certain curves in $D$ which is the a.s.\ limit of regularized versions of $e^{(\gamma/2) h(z)} \, |dz|$, where $|dz|$ is the Euclidean length element. It is shown in~\cite{shef-zipper} that $\nu_h$ can be defined on SLE$_\kappa$-type curves sampled independently from $h$ provided $\kappa \in (0,4)$ and $\gamma = \sqrt\kappa$.

The KPZ formula~\cite{kpz-scaling} relates the Euclidean fractal dimension and ``$\gamma$-quantum fractal dimension" of a random set $X\subset D$ independent from $h$. There are various rigorous versions of this formula using different notions of dimension; see~\cite{aru-kpz,grv-kpz,bjrv-gmt-duality,benjamini-schramm-cascades,wedges,shef-renormalization,shef-kpz,rhodes-vargas-log-kpz,ghm-kpz,gp-kpz}. Our main result will be proven by means of two versions of the KPZ formula, which will be used to express $\dim_{\mcl H} f_t^{-1}(Y)$ and $\dim_{\mcl H} Y$, respectively, in terms of the same quantum dimension. The first version of the KPZ formula (stated as Theorem~\ref{thm-length-kpz} below) relates the Euclidean dimension of a subset $X$ of an SLE$_\kappa$ curve $\eta$ for $\kappa \in (0,4)$ to the dimension of $\eta^{-1}(X)$, when $\eta$ is parameterized by $\gamma$-quantum length with respect to an independent GFF. This formula will be deduced from another KPZ formula, that of~\cite{ghm-kpz}. The other KPZ formula we will use directly is the boundary measure KPZ formula appearing in~\cite{rhodes-vargas-log-kpz}. 

Suppose $\kappa \not= 4$, $\eta$ is an SLE$_\kappa$-type curve, and $h$ is some variant of the GFF, independent from $\eta$. There is a natural quantum parameterization of $\eta$ with respect to $h$, which is different in each of the three phases of $\kappa$:
\begin{enumerate}
\item For $\kappa \in (0,4)$, we parameterize $\eta$ by $\gamma = \sqrt{\kappa}$-quantum length with respect to $h$. 
\item For $\kappa \in (4,8)$, we parameterize $\eta$ by $\gamma = 4/\sqrt{\kappa}$-quantum natural time with respect to $h$ (which is defined in~\cite{wedges} and reviewed in Section~\ref{sec-sle-kpz'} below). 
\item For $\kappa \geq 8$, we parameterize $\eta$ by $\gamma= 4/\sqrt{\kappa}$-quantum mass with respect to $h$. 
\end{enumerate}
The results of the present paper and~\cite{ghm-kpz} imply KPZ formulas for the dimension of $\eta^{-1}(X)$, where $X$ is a subset of $\eta$ which is independent from $h$, in each of the above three cases. Indeed, our Theorem~\ref{thm-length-kpz} is such a KPZ formula for SLE$_\kappa$ curve for $\kappa \in (0,4)$ equipped with the quantum length parameterization.  We will also prove in this paper an analogous KPZ formula in the case when $\kappa \in (4,8)$ and $\eta$ is parameterized by quantum natural time (Theorem~\ref{thm-length-kpz'}); the proof is very similar to that of Theorem~\ref{thm-length-kpz}. The KPZ formula~\cite[Theorem~1.1]{ghm-kpz} together with an absolute continuity argument immediately implies an analogous KPZ formula when $\kappa \geq 8$ and $h$ is parameterized by quantum mass with respect to $h$. 

We also remark that the recent paper \cite{zhan-nat-param} proves a Euclidean variant of Theorems~\ref{thm-length-kpz} and \ref{thm-length-kpz'}. The author shows that for an SLE$_\kappa$ $\eta$, $\kappa\in(0,8)$, with the natural parameterization, and any deterministic closed set $Y\subset\R$, the Hausdorff dimension of $\eta(Y)$ is a.s.\ equal to $1+\frac{\kappa}{8}$ times the Hausdorff dimension of $Y$.

\subsubsection*{Acknowledgements}
E.G.\ was supported by the U.S. Department of Defense via an NDSEG fellowship. N.H.\ was supported by a fellowship from the Norwegian Research Council. J.M.\ was partially supported by DMS-1204894.  The authors thank Ilia Binder, Greg Lawler, Scott Sheffield, and Xin Sun for helpful discussions.

\subsection{Main results}
\label{sec-main-results}

For $\kappa > 0$ and $d\in [0,1]$, define
\eqb \label{eqn-dim-function}
\Phi_\kappa(d) :=  \frac{1}{32\kappa} \left(4 + \kappa - \sqrt{(4 + \kappa)^2 - 16 \kappa d} \right) 
\left(12 + 3 \kappa + \sqrt{(4 + \kappa)^2 - 16 \kappa d } \right) .
\eqe  
We remark that $\Phi_\kappa=\Phi_{16/\kappa}$. 
Our first theorem relates the Hausdorff dimension of a deterministic set $Y\subset \BB R$ and the Hausdorff dimension of the subset of an SLE curve obtained by ``zipping up" the set $Y$ into the curve by means of a conformal map.

\begin{thm}  \label{thm-uniform-dim-time}
Let $\kappa  > 0$, $\kappa \not=4$, and let $\eta$ be a chordal SLE$_\kappa$ from $0$ to $\infty$ in $\BB H$, parameterized by half-plane capacity (resp.\ a radial SLE$_\kappa$ from $1$ to $0$ in $\BB D$, parameterized by minus log conformal radius, or a whole-plane SLE$_\kappa$ from $0$ to $\infty$ parameterized by capacity). Let $t > 0$ and let $f_t$ be the time $t$ centered Loewner map for $\eta$. Let $Y\subset \BB R$ (resp.\ $Y\subset \bdy\BB D$) be a deterministic Borel set. Almost surely, on the event $\{f_t^{-1}(Y) \subset \eta\}$ we have
\eqb \label{eqn-uniform-dim-time}
\dim_{\mcl H} f_t^{-1}( Y) = \Phi_\kappa\left( \dim_{\mcl H} Y \right)  
\eqe 
with $\Phi_\kappa$ as in~\eqref{eqn-dim-function}.
\end{thm}

\begin{remark} \label{remark-kpz}
For $\gamma \in (0,2)$ and $d\in [0,1]$, let
\eqb \label{eqn-kpz-function}
  \Psi_\gamma(d) := \left(1 + \frac{\gamma^2}{4} \right) d  -  \frac{\gamma^2}{4} d^2 
\eqe 
be the quadratic function appearing in the one dimensional KPZ formula. 
With $\gamma = \sqrt\kappa$ and $\Psi_\gamma$ as in~\eqref{eqn-kpz-function}, the function of~\eqref{eqn-dim-function} is given by $\Phi_\kappa(d) =2 \Psi_\gamma \left( \frac12 \Psi_\gamma^{-1}(d) \right)$, where $\Psi_\gamma^{-1}$ is chosen so as to take values in $[0,1]$. Our proof of Theorem~\ref{thm-uniform-dim-time} reflects this relationship. Indeed, we prove the theorem by expressing both $\dim_{\mcl H} f_t^{-1}(Y)$ and $\dim_{\mcl H} Y$ in terms of the $\gamma$-quantum dimension of $Y$ with respect to the length measure induced by a certain GFF coupled with $\eta$. This is done using the KPZ formulas of Theorem~\ref{thm-length-kpz} below and~\cite[Theorem~4.1]{rhodes-vargas-log-kpz}, respectively. Theorem~\ref{thm-length-kpz} is itself proven using~\cite[Theorem~1.1]{ghm-kpz}. 
\end{remark}

We also have a variant of Theorem~\ref{thm-uniform-dim-time} when we consider a complementary connected component of a whole SLE curve, rather than a curve stopped at a fixed time. We state the theorem for chordal SLE$_\kappa(\rho^L ; \rho^R)$ and whole-plane SLE$_\kappa(\rho)$ processes, which are constructed for $\rho^L , \rho^R > -2$ and $\rho >-2$ in~\cite[Section~2.2]{ig1} and~\cite[Section~2.1]{ig4}, respectively (see also Section~\ref{sec-prelim}). See Figure~\ref{fig1} for an illustration.

\begin{thm}  \label{thm-uniform-dim}
Let $\kappa \in (0,4) \cup (4,8)$, and suppose that we are in one of the following situations. 
\begin{itemize}
\item $D\subset \BB C$ is a simply connected domain which is not all of $\BB C$, $\rho^L , \rho^R >-2$, and $\eta$ is a chordal SLE$_\kappa(\rho^L ; \rho^R)$ in $D$ with some choice of starting point and target point and force points immediately to the left and right of the starting point. 
\item $D = \BB C$, $\rho > -2$, and $\eta$ is a whole-plane SLE$_\kappa(\rho)$ between two points in $\BB C\cup \{\infty\}$. 
\end{itemize}
Let $\mcl U$ be the set of connected components of $D \setminus \eta$. For $U\in \mcl U$, let $x_U$ (resp.\ $y_U$) be the point visited by $\eta$ at the time at which it starts (resp.\ finishes) tracing $\partial U$. Define the following subsets of $\mcl U$.
\begin{itemize}
\item $\mcl U^+$ (resp.\ $\mcl U^-$) is the set of $U \in \mcl U$ for which $x_U \not= y_U$ and the counterclockwise (resp.\ clockwise) arc of $\bdy U$ from $x_U$ to $y_U$ is traced by $\eta$. 
\item $\mcl U^0$ is the set of $U \in \mcl U$ for which $x_U = y_U$. 
\end{itemize}
For $U \in \mcl U^- \cup \mcl U^+$, let $g_U : \BB H \rta U$ be a conformal map which takes $0$ to $x_U$ and $\infty$ to $y_U$ (chosen in some manner which depends on $\eta$), and if $U \in \mcl U^0$ let $g_U : \BB H \rta U$ be a conformal map which takes $\infty$ to $x_U = y_U$. Let $Y\subset [0,\infty)$ be a deterministic Borel set. Then a.s.\ 
\begin{align} \label{eqn-uniform-dim}
\dim_{\mcl H} g_U(a Y) = \Phi_\kappa\left( \dim_{\mcl H} Y \right) , \quad &\text{for each $ U\in \mcl U^+$ and Lebesgue a.e.\ $a > 0$}, \notag\\
\dim_{\mcl H} g_U(-a Y) = \Phi_\kappa\left( \dim_{\mcl H}Y  \right) , \quad &\text{for each $ U\in \mcl U^-$ and Lebesgue a.e.\ $a > 0$}, \notag \\ 
\dim_{\mcl H} g_U(a Y + b) = \Phi_\kappa\left( \dim_{\mcl H} Y \right) , \quad &\text{for each $ U\in \mcl U^0$ and Lebesgue a.e.\ $(a,b ) \in \BB R^2$} ,  
\end{align}
where $\Phi_\kappa$ is as in~\eqref{eqn-dim-function}. 
\end{thm}

\begin{figure}[ht!]
	\begin{center}
		\includegraphics[scale=0.92]{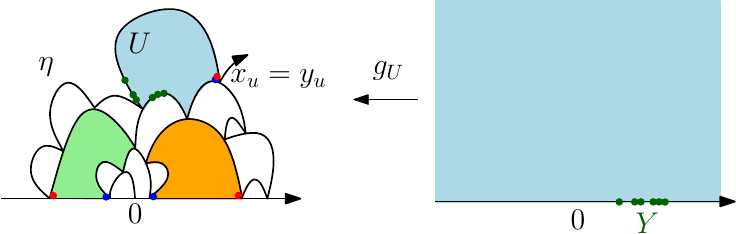}
	\end{center}
	\caption{\label{fig1} An illustration of the statement of Theorem~\ref{thm-uniform-dim} in the case of a chordal SLE$_\kappa$ $\eta$ and $\kappa\in(4,8)$. The theorem relates the Hausdorff dimension of the set $Y\subset[0,\infty)$ and the set $g_U(aY+b)$ for almost all $a,b>0$. The set $U$ on the left figure is contained in $\mcl U^0$, and the map $g_U:\BB H\to U$ is defined such that $g_U(\infty)=x_U=y_U$. The complementary connected component of $\eta$ shown in green (resp.\ orange) is contained in $\mcl U^+$ (resp.\ $\mcl U^-$). The blue (resp.\ red) dots represent the points $x_U$ (resp.\ $y_U$) for each of the three considered domains.} 
\end{figure}

In the setting of Theorem~\ref{thm-uniform-dim}, topological considerations imply that for $U\in\mcl U$, either the counterclockwise or clockwise arc of $\bdy U$ from $x_U$ to $y_U$ (or both) is traced by $\eta$, so $\mcl U = \mcl U^+ \cup \mcl U^- \cup \mcl U^0$. For $U\in \mcl U^+$ (resp.\ $U\in \mcl U^-$), the function $g_U$ maps $[0,\infty)$ (resp.\ $(-\infty , 0]$) to this counterclockwise (resp.\ clockwise) arc.  The set $\mcl U^0$ is empty if $\kappa \in (0,4)$ and for $\kappa > 4$ consists of ``bubbles" surrounded by either the left or right side of $\eta$ if $\kappa \in (4,8)$.

For $U\in \mcl U^+  $, the set of conformal maps $\BB H\rta U$ which take $0$ to $x_U$ and $\infty$ to $y_U$ is the same as the set of maps of the form $z\mapsto g_U(a z)$ for $a > 0$. Hence a.s.\ $\dim_{\mcl H} g (  Y) = \Phi_\kappa\left( \dim_{\mcl H} Y \right)$ for ``Lebesgue a.e." conformal map $g : \BB H \rta U$ taking $0$ to $x_U$ and $\infty$ to $y_U$. Similar statements hold for $\mcl U^-$ and $\mcl U^0$. We leave it as an open problem to determine whether this relation in fact holds a.s.\ for every such conformal map simultaneously (see Problem~\ref{item-all-map} in Section~\ref{sec-open-problems}). We note, however, that~\eqref{eqn-uniform-dim} (and the analogous relation in Theorem~\ref{thm-uniform-dim-time}) does not hold a.s.\ for all choices of $Y$ simultaneously. Indeed, taking $Y$ to be one of the multifractal spectrum sets $\wt\Theta^s \subset \BB R$ for $s \in [-1,1]$ studied in~\cite{gms-mf-spec} gives a counterexample.

\begin{remark}
Theorem~\ref{thm-uniform-dim-time} can be used to give another derivation of the double point dimension of SLE$_\kappa$ for $\kappa \in (4,8)$.  Indeed, the double points of such an SLE correspond to intersection points with the boundary which are subsequently mapped into the domain by the reverse Loewner flow.  The dimension of the intersection of an SLE$_\kappa$ for $\kappa \in (4,8)$ with the domain boundary was shown to be $2-8/\kappa$ in \cite{alberts-shef-bdy-dim} and $\Phi_\kappa(2-8/\kappa) = 2-(12-\kappa)(4+\kappa)/(8\kappa)$, which is the dimension of the double points of SLE$_\kappa$ \cite{miller-wu-dim}.

The cut point dimension of SLE$_\kappa$ for $\kappa \in (4,8)$ can similarly be derived using Theorem~\ref{thm-uniform-dim}.  Indeed, it is shown in \cite{ig1} that the conditional law of the left boundary of a chordal version of such an SLE given its right boundary is that of an SLE$_{16/\kappa}(16/\kappa-4;-8/\kappa)$ process.  The dimension of the intersection of an SLE$_{16/\kappa}(16/\kappa-4;-8/\kappa)$ with $[0,\infty)$ was shown in \cite{miller-wu-dim} to be $5-8/\kappa + \kappa/2$ and $\Phi_\kappa(5-8/\kappa + \kappa/2) = 3-3\kappa/8$, which is the dimension of the cut points of an SLE$_\kappa$ derived in \cite{miller-wu-dim}.
\end{remark}

\subsection{SLE/LQG background} 
\label{sec-prelim}

In this subsection we briefly review some facts about SLE and LQG which will be needed for the proofs of our main results. We refer to the cited papers for more details. See also~\cite[Sections~1.2 and~1.4]{ghm-kpz} for a more detailed overview.

We first recall the definition of chordal and radial SLE$_\kappa(\ul\rho)$ for $\kappa >0$ and a finite vector of weights $\ul\rho = (\rho_1 , \dots, \rho_n)$ and force points $x_1 , \dots , x_n$ in the closure of the domain. Such processes were first introduced in~\cite[Section~8.3]{lsw-restriction}.  See also \cite{sw-coord} and~\cite[Section~2.2]{ig1}.  As defined in \cite[Section~2.2]{ig1}, the \emph{continuation threshold} for an SLE$_\kappa(\ul\rho)$ is the first time that the sum of the weights of the force points which have been disconnected from the target point by the curve is $\leq -2$. This time was defined in~\cite[Section~2.2]{ig1}, and is the largest time up to which SLE$_\kappa(\ul\rho)$ is defined as a continuous curve. Note that the continuation threshold may be infinite. We also recall the definition of whole-plane SLE$_\kappa(\rho)$ for $\rho  > -2$~\cite[Section~2.1]{ig4}.  

Now fix $\gamma \in (0,2)$. A \emph{Liouville quantum gravity surface} is an equivalence class of pairs $(D , h )$ consisting of a domain $D\subset \BB C$ and a distribution $h$ on $D$, with two such pairs $(D, h)$ and $(\wt D , \wt h)$ declared to be equivalent (i.e.\ different parameterizations of the same surface) if there is a conformal map $\phi : \wt D \rta D$ such that
\eqb \label{eqn-lqg-coord}
\wt h   = h \circ \phi  +Q \log |\phi'| ,\quad \text{for} \quad Q = \frac{2}{\gamma} + \frac{\gamma}{2}   .
\eqe 
We refer to the distribution $h$ as an \emph{embedding} of the quantum surface.

The quantum area and length measures $\mu_h$ and $\nu_h$ of~\cite{shef-kpz} are preserved under transformations of the form~\eqref{eqn-lqg-coord}, in the sense that it is a.s.\ the case that for each Borel set $A\subset \wt D$, we have $\mu_h(\phi(A)) = \mu_{\wt h}(A)$ and similarly for $\nu_h$; see~\cite[Proposition~2.1]{shef-kpz}.  Hence these measures are well-defined on the LQG surface $(D, h)$. One can also define quantum surfaces with $k\in\BB N$ marked points in $D\cup\bdy D$ by requiring that the conformal map $\phi$ takes the marked points for one surface to those for the other. 

For $\alpha < Q$, an \emph{$\alpha$-quantum cone} is an infinite-volume doubly-marked quantum surface parameterized by $\BB C$ which describes the local behavior of $\wt h- \alpha \log |\cdot|$ near 0, where $\wt h$ is a whole-plane GFF (see~\cite[Section~4.2]{wedges} for a precise definition).  Similarly, an $\alpha$-quantum wedge is an infinite-volume quantum surface parameterized by $\BB H$, which describes the behavior of $\wt h - \alpha \log |\cdot|$ near 0, for $\wt h$ a free-boundary GFF on $\BB H$ (see~\cite[Section~1.6]{shef-zipper} or~\cite[Section~4.2]{wedges}). As explained in~\cite[Section~4.4]{wedges}, one can also define an $\alpha$-quantum wedge for $\alpha\in (Q,Q+\gamma/2)$. In this case, the wedge is not parameterized by $\BB H$ and instead consists of an infinite ordered sequence of finite-volume ``beads", each of which has the topology of the disk and has finite quantum area and boundary length. 

Quantum disks are finite-volume quantum surfaces parameterized by the unit disk $\BB D$ (or equivalently any simply connected domain in $\BB C$), which are most often taken to have one or two marked boundary points (which are sampled uniformly from the quantum boundary measure). One can consider quantum disks with specified area, boundary length, or both. Quantum spheres are finite-volume quantum surfaces parameterized by the Riemann sphere, often taken to have fixed area and sometimes taken to have one, two, or three marked points. See~\cite[Section~4.5]{wedges}.  
 
The quantum surfaces introduced above can be embedded in various ways into $\BB C$ (in the case of a cone or a sphere) or $\BB H$ (in the case of a thick wedge, a bead of a thin wedge, or a disk). The \emph{circle average embedding}, which we will define just below, is a particularly convenient choice of embedding for the quantum surfaces considered above, since with this choice of embedding, the law of the field is absolutely continuous with respect to the law of a free-boundary or whole-plane GFF with a particular choice of additive constant on any domain which is bounded away from the origin, infinity, and $\partial\BB D$. For any $r>0$ and any field $h$ on $\BB C$ (resp.\ $\BB H$) we let $h_r(0)$ be the average of $h$ around $\bdy B_r(0)$ (resp.\ $\bdy B_r(0)\cap\BB H$)~\cite[Section 3]{shef-kpz}. 	
 \begin{defn}
 	The \emph{circle average embedding} of the quantum surfaces introduced above is defined as follows.
 	\begin{itemize}
 		\item For a quantum cone (resp.\ thick wedge) the circle average embedding is the distribution $h$ on $\BB C$ (resp.\ $\BB H$) such that 1 is the largest $r > 0$ for which $h_r(0)+Q\log(r)\leq 0$. 
 		\item For a quantum sphere (resp.\ quantum disk or bead of a thin quantum wedge) the circle average embedding is the distribution $h$ on $\BB C$ (resp.\ $\BB H$) such that the function $r\mapsto h_r(0)+Q\log(r)$ attains its maximum at $r=1$.
 	\end{itemize}
 	\label{def:circleaverage}
 \end{defn}	
 The circle average embedding was the embedding used when defining the mentioned quantum surfaces in \cite{wedges}.

In this paper our main interest in the above quantum surfaces stems from their relationship with SLE$_\kappa$. If one cuts an $\alpha$-quantum cone by an independent whole-plane SLE$_\kappa(\rho)$ curve for $\kappa = \gamma^2 \in (0,4)$ and appropriate $\rho > -2$ depending on $\alpha$, then one obtains an $\alpha'$-quantum wedge for a certain value of $\alpha'$ depending on $\alpha$~\cite[Theorem~1.5]{wedges}. Similarly, one can cut an $\alpha$-quantum wedge by an independent chordal SLE$_\kappa(\rho^L ; \rho^R)$ curve (with force points immediately to the left side and the right side, respectively, of the starting point of the curve) to get a pair of independent quantum wedges~\cite[Theorem~1.2]{wedges} for certain $\rho^L,\rho^R>-2$. In the case when $\alpha\in (Q,Q+\gamma/2)$, one replaces the SLE$_\kappa(\rho^L ; \rho^R)$ curve with a concatenation of such curves, one in each bead.

If $\gamma \in (\sqrt 2 , 2)$ and one cuts an $\alpha$-quantum cone by an independent SLE$_\kappa(\rho)$ curve for $\kappa= 16/\gamma^2 \in (4,8)$ and appropriate $\rho > -2$, one gets three independent beaded quantum surfaces corresponding to the complementary connected components of the curve whose boundaries are traced by the left side of the curve, the right side of the curve, and both sides of the curve, respectively. One of these surfaces is an $\alpha'$-quantum wedge for $\alpha' > Q$ and the other two are L\'evy trees of quantum disks~\cite[Theorem~1.17]{wedges}. A similar statement holds for an $\alpha$-quantum wedge cut by an independent SLE$_\kappa(\rho^L ;\rho^R)$ for $\kappa = 16/\gamma^2$ and appropriate $\rho^L,\rho^R>-2$~\cite[Theorem~1.16]{wedges}. 

For $\kappa >4$ whole-plane space-filling SLE$_{\kappa }$ from $\infty$ to $\infty$ is a variant of SLE$_{\kappa }$ which fills all of $\BB C$, introduced in~\cite[Sections~4.3 and~1.2.3]{ig4} (see~\cite[Section~1.4.1]{wedges} for the whole-plane case). In the case when $\kappa  \geq 8$, so SLE$_{\kappa }$ is already space-filling, whole-plane space-filling SLE$_{\kappa}$ is a two-sided version of chordal SLE$_{\kappa}$. In the case when $\kappa \in (4,8)$, whole-plane space-filling SLE$_{\kappa}$ is obtained by iteratively filling in each of the bubbles disconnected from $\infty$ by a two-sided variant of chordal SLE$_{\kappa}$ with a space-filling SLE$_{\kappa}$ loop (so in particular cannot by described by the Loewner equation). It is immediate from the construction that the marginal law of the left (resp.\ right) boundary of $\eta'$ stopped upon hitting a fixed point $z\in\BB C$ is that of a whole-plane SLE$_{16/\kappa}(2-16/\kappa)$ from $z$ to $\infty$ (c.f.~\cite[Theorem~1.1]{ig4}). 

Suppose $(\BB C , h , 0,  \infty)$ is a $\gamma$-quantum cone, $\gamma \in (0,2)$, and $\eta'$ is an independent whole-plane space-filling SLE$_{\kappa}$, parameterized by quantum mass with respect to $h$ (so that $\mu_h(\eta'([s,t])) = t-s$ for $s <t$).  For $t \in \BB R$, let $L_t$ (resp.\ $R_t$) be the change in the quantum length of the left (resp.\ right) outer boundary of $\eta'((-\infty, t])$ relative to time $0$. Then $Z_t = (L_t , R_t)$ is a correlated two-dimensional Brownian motion with correlation $-\cos(4\pi/\kappa)$~\cite[Theorem~1.9]{wedges}.  (The formula $-\cos(4\pi / \kappa)$ for the correlation of the Brownian motion was only proved for $\kappa \in (4,8]$ in \cite{wedges}.  That the same formula holds for $\kappa > 8$ was established in \cite{kappa8-cov}.)  Furthermore, $Z$ a.s.\ determines $(h ,\eta')$ modulo rotation~\cite[Theorem~1.11]{wedges}. Many quantities associated with $\eta'$ can be described explicitly in terms of $Z$. For example, the curve $\eta'$ hits the left (resp.\ right) outer boundary of $\eta'((-\infty ,0])$ precisely when $L$ (resp.\ $R$) hits a running infimum relative to time 0. See~\cite{ghm-kpz} for further examples. The above facts are collectively referred to as the \emph{peanosphere description} of $(h , \eta')$.

\section{Proofs}
\label{sec-proof}

\subsection{KPZ formula for quantum lengths along an SLE$_\kappa$ curve for $\kappa \in (0,4)$}
\label{sec-length-kpz}
 
In this subsection we will prove a KPZ formula for the quantum length measure along an SLE$_\kappa$ curve for $\kappa \in (0,4)$, which is needed for our proofs of Theorems~\ref{thm-uniform-dim-time} and~\ref{thm-uniform-dim}. We state the theorem at a high level of generality: we allow for chordal, radial, and whole-plane SLE$_\kappa$ (possibly with force points); and quantum lengths measured with respect to a free- or zero-boundary GFF, as well as with respect to the distributions which parameterize the various quantum surfaces defined in~\cite[Section 4]{wedges} (which are variants of the GFF). See Section~\ref{sec-prelim} for a brief review of the objects involved in the theorem statement; see in particular Definition~\ref{def:circleaverage} for the definition of the circle average embedding.

\begin{thm}  \label{thm-length-kpz}
Let $\kappa\in (0,4)$, let $\ul\rho$ be a finite vector of weights, and let $\eta$ be a chordal or radial SLE$_\kappa(\ul\rho)$ in a simply connected domain $D\subsetneq \BB C$, with some choice of starting point, target point and force points, stopped when it hits the continuation threshold  (which may be infinite). Let $h$ be a free-boundary GFF on $D$ (with the additive constant fixed in some arbitrary way) independent from $\eta$ and let $\nu_h$ be the $\gamma$-quantum length measure induced by $h$, $\gamma = \sqrt\kappa$. Suppose that $\eta$ is parameterized by $\nu_h$-length, so that $\nu_h(\eta([0,t])) = t$ for each $t > 0$. Let $X$ be a random subset of $\eta\setminus \bdy D$ which is independent from $h$. Then almost surely,
\eqb \label{eqn-length-kpz}
\dim_{\mcl H} X  
= \left( 1 + \frac{\gamma^2}{4} \right)\dim_{\mcl H} \eta^{-1}(X) - \frac{\gamma^2}{8} (\dim_{\mcl H} \eta^{-1}(X) )^2  .
\eqe   
The same holds if we make one or both of the following changes (where $Q = 2/\gamma + \gamma/2$ is as in~\eqref{eqn-lqg-coord}). 
\begin{itemize}
\item Replace $\eta$ by a whole-plane SLE$_\kappa(\rho)$ for $\rho > -2$ and replace $h$ by a whole-plane GFF. 
\item In the chordal or radial case, replace $h$ by a zero-boundary GFF on $D$, or for $D=\BB H$ replace $h$ by the circle average embedding into $\BB H$ of a quantum disk (with fixed area, boundary length, or both), an $\alpha$-quantum wedge for $\alpha \leq Q $, or a single bead of an $\alpha$-quantum wedge for $\alpha\in(Q,Q+\gamma/2)$. In the whole-plane case, replace $h$ by the \emph{circle average embedding into $\BB C$ of} a quantum sphere, or an $\alpha$-quantum cone for $\alpha < Q $. 
\end{itemize}
\end{thm}

Theorem~\ref{thm-length-kpz} will be proven using the KPZ-type formula~\cite[Theorem~1.1]{ghm-kpz} together with~\cite[Theorem~1.9]{wedges} and 
a version of Kaufman's theorem for subordinators~\cite[Theorem~4.1]{hawkes-uniform}. We remark that an analogue of Theorem~\ref{thm-length-kpz} when $\eta$ is a flow line of $h$ (in the sense of~\cite{ig1,ig2,ig3,ig4}) instead of an SLE curve independent from $h$ is proven in~\cite{aru-kpz}.

\begin{remark}
The right side of~\eqref{eqn-length-kpz} is equal to $2 \Psi_\gamma\left( \frac12 \dim_{\mcl H} \eta^{-1}(X) \right)$, with $\Psi_\gamma$ as in~\eqref{eqn-kpz-function}.  
\end{remark}

We will first prove Theorem~\ref{thm-length-kpz} in the special case of whole-plane SLE$_\kappa(2-\kappa)$ on an independent $\gamma$-quantum cone. This case is particularly convenient because it exactly fits into the framework of the peanosphere construction (Section~\ref{sec-prelim}), which is also the setting of~\cite[Theorem~1.1]{ghm-kpz}. 

\begin{lem} \label{prop-length-kpz0}
Let $\kappa \in (0,4)$ and let $\eta$ be a whole-plane SLE$_\kappa(2-\kappa)$ from $0$ to $\infty$. Let $(\BB C ,h , 0 ,\infty)$ be a $\gamma$-quantum cone independent from $\eta$ with the circle average embedding and let $\nu_h$ be its $\gamma$-quantum length measure, $\gamma = \sqrt\kappa$. Suppose that $\eta$ is parameterized by $\nu_h$-length, so that $\nu_h(\eta([0,t])) = t$ for each $t > 0$. If $X\subset \eta$ is a set which is independent from $h$, then $\dim_{\mcl H} X$ and $\dim_{\mcl H}\eta^{-1}(X)$ are a.s.\ related by the formula~\eqref{eqn-length-kpz}. 
\end{lem}
\begin{proof} 
Let $\eta'$ be a whole-plane space-filling SLE$_{16/\kappa}$ from $\infty$ to $\infty$, independent from $h$ and parameterized by quantum mass with respect to $h$ in such a way that $\eta'(0) = 0$. By the construction in~\cite[Section 1.4.1]{wedges}, the right outer boundary of $\eta'((-\infty , 0])$ is the flow line from $0$ to $\infty$ of a certain whole-plane GFF, so by~\cite[Theorem~1.1]{ig4} it has the law of a whole-plane SLE$_\kappa(2-\kappa)$ curve. Therefore we can couple $\eta$ with $(\eta' , h)$ in such a way that $\eta$ is a.s.\ equal to this right outer boundary.  In this coupling $\eta$ is determined by~$\eta'$ (viewed modulo parameterization) and hence is independent from $h$. We assume that $\eta$ is parameterized by $\gamma$-quantum length with respect to $h$.  

For $t \geq 0$, let $R_t$ be the change in the right boundary length of $\eta'$ between time $0$ and time $t$, as in~\cite[Theorem~1.9]{wedges}. That theorem tells us that $R$ has the law of a deterministic constant multiple of a standard linear Brownian motion. For $r \geq 0$, let
\eqbn
S_r := \inf\left\{t \geq 0 \,:\, R_t =-r\right\} .
\eqen
Equivalently, $S_r$ is the first time that $\eta'$ covers up $r$ units of quantum length along its right boundary, or, a.s.\ for each fixed $r$, the infimum of the times for which $\eta(r)$ is no longer on the left right boundary of $\eta'$. In particular, since $\eta$ is parameterized by quantum length we have $\eta'(S_r) = \eta(r)$. 

The process $S$ has the law of a stable subordinator of index $1/2$ (see e.g., \cite[Section~2.2]{bertoin-sub}). By~\cite[Theorem~4.1]{hawkes-uniform}, we a.s.\ have
\eqbn
\dim_{\mcl H} S(A)   = \frac12 \dim_{\mcl H} A ,\quad \text{for every Borel set $A\subset [0,\infty)$ simultaneously.}
\eqen
In particular, if $X\subset \eta$ is chosen in a manner which is independent from $h$, then a.s.\ 
\eqbn
\dim_{\mcl H} (\eta')^{-1}(X) =     \frac12  \dim_{\mcl H} (\eta' \circ S)^{-1} (X) = \frac12  \dim_{\mcl H} \eta^{-1} (X)  .
\eqen
By~\cite[Theorem~1.1]{ghm-kpz}, we a.s.\ have
\eqbn
\dim_{\mcl H} X = \left(2 + \frac{\gamma^2}{2} \right) \dim_{\mcl H} (\eta')^{-1}(X) - \frac{\gamma^2}{2} (\dim_{\mcl H} (\eta')^{-1}(X))^2 .
\eqen
Combining these relations yields the statement of the lemma.  
\end{proof}

Now we will deduce the general case of Theorem~\ref{thm-length-kpz} from Lemma~\ref{prop-length-kpz0} and various elementary Markov property and absolute continuity arguments. 

\begin{proof}[Proof of Theorem~\ref{thm-length-kpz}]
In the cases of the whole-plane GFF and the free-boundary GFF we may fix the additive constant arbitrarily since changing the additive constant corresponds to multiplying all quantum lengths by a constant, hence $\dim_{\mcl H} \eta^{-1}(X)$ is left unchanged. By stability of Hausdorff dimensions under countable unions and by absolute continuity of the fields in domains bounded away from zero, infinity and $\partial \BB D$, the cases of the quantum cone, the quantum wedge, the quantum sphere and the quantum disk reduce to the case of the GFF in $\BB C$ or $\BB H$.	By conformal invariance and the LQG coordinate change formula (see~\eqref{eqn-lqg-coord}) we only need to prove the statement for each chordal or radial SLE$_\kappa(\ul\rho)$ process in a single choice of domain~$D$.

First consider the case where $\eta$ is a radial SLE$_{\kappa}(2-\kappa)$ from $1$ to $\infty$ in $\BB C\setminus \ol{\BB D}$, with force point located at $z \in \bdy\BB D\setminus \{1\}$. Let $\wt\eta$ be a whole-plane SLE$_{\kappa}(2-\kappa)$ from $0$ to $\infty$ parameterized by capacity seen from $\infty$ and let $\tau$ the the smallest $t \geq 0$ for which the centered Loewner map $f_t$ from the unbounded component of $\BB C\setminus \wt\eta((-\infty , t])$ to $\BB C\setminus \ol{\BB D}$ maps the force point of $\wt\eta$ to $z$. Note that it follows from scale invariance and the domain Markov property that $\tau < \infty$ a.s. By the domain Markov property we can couple $\eta$ with $\wt\eta$ in such a way that $\eta = f_\tau(\wt\eta|_{[\tau,\infty)})$ a.s.  

Let $\wt h$ be a whole-plane GFF independent from $\wt\eta$ and let $h' := \wt h \circ f_\tau^{-1}  + Q\log |(f_\tau^{-1})'|$. By the Markov property and conformal invariance of the GFF, the conditional law of $h'$ given $\wt\eta([-\infty,\tau])$ and $\wt h|_{\wt\eta([-\infty,\tau])}$ is that of a zero-boundary GFF $h$ on $\BB C\setminus \ol{\BB D}$ plus a function which is harmonic on $\BB C\setminus \ol{\BB D}$. If $X \subset \eta$ is determined by~$\eta$, viewed modulo parameterization, then $f_t^{-1}(X)$ is a subset of $\wt\eta$ which is independent from $\wt h$. By Lemma~\ref{prop-length-kpz0} and local absolute continuity, the formula~\eqref{eqn-length-kpz} holds a.s.\ with $f_t^{-1}(X)$ in place of $X$, $\wt\eta$ in place of $\eta$, and $\wt h$ in place of $h$. By the LQG coordinate change formula we can apply the map $f_\tau$ to obtain~\eqref{eqn-length-kpz} with $h'$ in place of $h$. Since $\nu_h$ and $\nu_{h'}$ differ by multiplication by a smooth function, we obtain~\eqref{eqn-length-kpz} for $X$, $\eta$, and $h$. 

Now suppose that $\eta$ is a chordal or radial SLE$_\kappa(\ul\rho)$ in $\BB D$ started from 1, with arbitrary choice of target point, weights, and force points located at positive distance from $1$, stopped when it hits the continuation threshold. By the Schramm-Wilson coordinate change formula~\cite[Theorem~3]{sw-coord} we immediately reduce to the case of radial SLE$_\kappa(\ul\rho)$. Let $h$ be a zero-boundary GFF on $\BB D$ independent from $\eta$. Fix some $z \in \bdy \BB D\backslash \{1\}$.  
Let $V\subset \BB D$ be a simply connected subdomain such that $\bdy V \cap \bdy \BB D$ contains a neighborhood of $1$ in $\bdy \BB D$ and $V$ lies at positive distance from the target point and all of the force points of $\eta$ and from $z$. Let $\tau_V$ be the exit time of $\eta$ from $V$. 
By the form of the Loewner driving function for general radial SLE$_\kappa(\ul\rho)$, we find that the law of $\eta|_{[0,\tau_V]}$ is absolutely continuous with respect to the law of a radial $\op{SLE}_\kappa(2-\kappa)$ in $\BB D$, started from 1, targeted at a point at positive distance from $V$, with force point at $z$, stopped at the first time it exits $V$.   
Therefore, the statement of the theorem for $\eta$ follows from the statement for radial SLE$_\kappa(2-\kappa)$ (proven just above) provided we require that $X\subset \eta([0,\tau_V])$.

Now consider the case where $\eta$ is a chordal or radial SLE$_\kappa(\ul\rho)$ in $\BB D$ starting from $1$ with completely arbitrary choices of target point, weights, and force points (even force points precisely on either side of the starting point), stopped when it hits the continuation threshold. Let $(f_t)_{t\geq 0}$ be the centered Loewner maps for $\eta$. For $\ep > 0$, let $\tau_0^\ep = \sigma_0^\ep = 0$. Inductively, if $k\in\BB N$ and $\tau_{k-1}^\ep$ and $\sigma_{k-1}^\ep$ have been defined, let $\tau_k^\ep$ be the minimum of $\ep^{-1}$ and the smallest $t > \sigma_{k-1}^\ep$ for which the driving function $W_t$ lies at distance at least $\ep$ from the image of each of the force points of $\eta$ under $f_t$ and let $\sigma_k^\ep$ be the minimum of $\ep^{-1}$ and the smallest $t > \tau_k^\ep$ for which $W_t$ lies within distance $\ep/2$ of at least one of the images of the force points of $\eta$ under $f_t$. Note that each $\tau_k^\ep$ and $\sigma_k^\ep$ is a stopping time for $\eta$. By the domain Markov property of SLE$_\kappa(\ul\rho)$, the LQG coordinate change formula, and the preceding paragraph, we find that the statement of the corollary holds for $\eta$ provided we require that $X \subset \eta([\tau_k^\ep , \sigma_k^\ep])$ for some $k\in\BB N$. Taking a limit as $\ep \rta 0$ and using countable stability of Hausdorff dimension yields the statement of the theorem in the case of general chordal or radial SLE$_\kappa(\ul\rho)$. 

Finally, the case of whole-plane SLE$_\kappa( \rho)$ for $\rho > -2$ with $\rho\not=2-\kappa$ follows from the case of radial SLE$_\kappa( \rho)$ and an argument as in the case of radial SLE$_\kappa(2-\kappa)$. 
\end{proof}

\subsection{KPZ formula for quantum natural time of an SLE$_\kappa$ curve for $\kappa \in (4,8)$} 
\label{sec-sle-kpz'}

In this subsection we prove a variant of Theorem~\ref{thm-length-kpz} for SLE$_\kappa$ with $\kappa \in (4,8)$. 

If $\eta$ is some version of SLE$_{\kappa}$ for $\kappa \in (4,8)$ and $h$ is some variant of the GFF, then the natural quantum parameterization of $\eta$ with respect to $h$ is called the \emph{quantum natural time}. This parameterization is defined in~\cite[Definition 6.23]{wedges} in the case when $\eta$ is an ordinary whole-plane, chordal, or radial SLE$_{\kappa'}$ and $h$ is a whole-plane or free-boundary GFF plus a certain log singularity. In this case, the quantum surfaces parameterized by the bubbles disconnected from the target point by $\eta$ can be described by a Poisson point process parameterized by $\BB R$, and the quantum natural time of $\eta$ is the time parameterization corresponding to this Poisson point process. Note that the quantum natural time parameterization of a given segment of $\eta$ is determined by the restriction of $h$ to an arbitrary small neighborhood of that segment, since it depends only on the quantum areas or lengths of the small bubbles cut out by that segment of $\eta$. 
Hence quantum natural time in the case of an SLE$_{\kappa }(\ul\rho)$ process and a distribution which locally looks like a free-boundary GFF can be defined using local absolute continuity.

\begin{thm}  \label{thm-length-kpz'}
Let $\kappa\in (4,8)$, let $\ul\rho$ be a finite vector of weights, and let $\eta$ be a chordal or radial SLE$_\kappa(\ul\rho)$ in a simply connected domain $D\subsetneq \BB C$ with some choice of starting point, target point and force points, stopped when it hits the continuation threshold  (which may be infinite). Let $h$ be a free-boundary GFF on $D$ (with the additive constant fixed in some arbitrary way) independent from $\eta$. Suppose that $\eta$ is parameterized by $\gamma$-quantum natural time with respect to $h$, $\gamma = 4/\sqrt\kappa$. Let $X$ be a random subset of $\eta\setminus \bdy D$ which is independent from $h$. Then almost surely, 
\eqb \label{eqn-length-kpz'}
\dim_{\mcl H} X  
= \left( 1 + \frac{4}{\gamma^2} \right)\dim_{\mcl H} \eta^{-1}(X) - \frac{2}{\gamma^2} (\dim_{\mcl H} \eta^{-1}(X) )^2  .
\eqe   
The same holds if we make one or both of the following changes (where $Q = 2/\gamma + \gamma/2$ is as in~\eqref{eqn-lqg-coord}). 
\begin{itemize}
\item Replace $\eta$ by a whole-plane SLE$_\kappa(\rho)$ for $\rho > -2$ and replace $h$ by a whole-plane GFF. 
\item In the chordal or radial case, replace $h$ by a zero-boundary GFF on $D$, or for $D=\BB H$ replace $h$ by the circle average embedding into $\BB H$ of a quantum disk (with fixed area, boundary length, or both), an $\alpha$-quantum wedge for $\alpha \leq Q $, or a single bead of an $\alpha$-quantum wedge for $\alpha\in(Q,Q+\gamma/2)$. In the whole-plane case, replace $h$ by circle average embedding into $\BB C$ of a quantum sphere, or an $\alpha$-quantum cone for $\alpha < Q $.
\end{itemize}
\end{thm}
\begin{proof}
As in the proof of Lemma~\ref{thm-length-kpz} we first treat a single special case using~\cite[Theorem~4.1]{hawkes-uniform} and~\cite[Theorem 1.1]{ghm-kpz} and then extend to the other SLE$_{\kappa }$-type processes and GFF-type distributions in the theorem statement using local absolute continuity. 

We start with the case when $\eta$ is a whole-plane SLE$_\kappa(\kappa -6)$ from 0 to $\infty$ and $h$ is the circle average embedding of a $\gamma$-quantum cone. 
Let $\eta'$ be a whole-plane space-filling SLE$_{\kappa}$ from $\infty$ to $\infty$, independent from $h$ and parameterized by $\gamma$-quantum mass with respect to $h$ in such a way that $\eta'(0) = 0$. 
For $t \geq 0$, let $L_t$ and $R_t$ be the change in the left and right quantum boundary lengths of $\eta'$ with respect to $h$ between time $0$ and time $t$, as in~\cite[Theorem~1.9]{wedges} and let $Z = (L,R)$. That theorem tells us that $Z$ has the law of a correlated two-dimensional Brownian motion with correlation $-\cos(4\pi/\kappa)$. 

Following~\cite[Section 1.4.2]{wedges}, we say that a time $t \in [0,\infty)$ is \emph{ancestor free} if there does not exist $s\in [0,t]$ such that $L_s = \inf_{r \in [s,t]} L_r$ and $R_s = \inf_{r\in [s,t]} R_r$.  Let $\mcl A\subset[0,\infty)$ denote the set of ancestor free times, and for any $t\geq 0$ let $\theta_t:[0,\infty)\to[t,\infty)$ denote the shift operator.
First we claim that the set $\mcl A$ of ancestor free times is a regenerative set, i.e., we claim that for any stopping time $S$ for the filtration generated by $(\mcl A\cap[0,s])_{s\geq 0}$ for which $S\in\mcl A$ a.s., the set $\mcl A\circ\theta_{S}=\{s\geq 0\,:\,s+S\in\mcl A \}$ has the same distribution as $\mcl A$ and is independent of $\mcl A\cap[0,t]$. Our claim is immediate by the strong Markov property of Brownian motion, since any such stopping time $S$ is also a stopping time for $Z$. 

Since $\mcl A$ is regenerative, it can be parametrized by a local time (see \cite[Section 2.1]{bertoin-sub} and the text above \cite[Proposition 1.13]{wedges}). Let $s\mapsto T(s)$ be the right continuous inverse local time of the ancestor free times of $Z$ relative to time 0, as in~\cite[Proposition 10.3]{wedges} and let $ \eta(s) := \eta'(T(s))$ for $s\geq 0$. By~\cite[Lemma~10.4]{wedges}, the time reversal of $ \eta$ is the counterflow line (in the sense of~\cite{ig4}) from $\infty$ to 0 of the whole-plane GFF used to construct $\eta'$ which travels through $\eta'([0,\infty))$. By the discussion just after~\cite[Theorem~1.6]{ig4}, this counterflow line is a whole-plane SLE$_{\kappa }(\kappa -6)$ process from $\infty$ to 0 so by reversibility~\cite[Theorem~1.20]{ig4}, $\eta$ has the law of a whole-plane SLE$_\kappa(\kappa-6)$ from 0 to $\infty$. We see from the proof of~\cite[Proposition 10.3]{wedges} that $\eta$ is parameterized by quantum natural time (up to multiplication by a deterministic constant), since it follows from this proof that the local time at the ancestor free times of $Z$ can be obtained by counting the number of bubbles enclosed by $\eta$ with quantum boundary length in an interval $[2^{-(k+1)},2^{-k}]$ for $k\in\BB N$, normalizing appropriately, and sending $k\rta\infty$. It is immediate from \cite[Definition 6.23]{wedges} that the same property holds for the quantum natural time of $\eta$.

The range of $T$ is the set $\mcl A$ of ancestor free times of $Z$. The law of $\mcl A$ is scale invariant because of the scale invariance of $Z$. The time at which the local time at $\mcl A$ exceeds any given level $s$ is a stopping time for $Z$, which implies by the strong Markov property of $Z$ that $T\circ\theta_s$ has the same law as $T$ and is independent of $T|_{[0,s]}$. Therefore $T$ has independent stationary increments. By~\cite[Lemma~1.11 and Theorem~3.2]{bertoin-sub}, $T$ is a stable subordinator. By~\cite[Example 2.3]{ghm-kpz}, the Hausdorff dimension of the ancestor free times of $Z$ is a.s.\ equal to $\kappa /8$, so the index of $T$ is $\kappa/8$. By~\cite[Theorem~4.1]{hawkes-uniform}, if $X$ is a subset of $ \eta$, then a.s.\ 
\eqbn
\dim_{\mcl H} (\eta')^{-1}(X) =  \frac{\kappa}{8} \dim_{\mcl H}  \eta^{-1} (X)  .
\eqen
By combining this with~\cite[Theorem~1.1]{ghm-kpz} as in the proof of Theorem~\ref{thm-length-kpz}, we obtain the theorem in the case of a whole-plane SLE$_\kappa(\kappa-6)$ and a $\gamma$-quantum cone with the circle average embedding. The general case follows from this special case and an absolute continuity argument as in the proof of Theorem~\ref{thm-length-kpz}.
\end{proof}

\subsection{Proof of Theorems~\ref{thm-uniform-dim-time} and~\ref{thm-uniform-dim}}
\label{sec-main-proof}

In this subsection we combine the KPZ formulas of Theorem~\ref{thm-length-kpz} and~\cite[Theorem~4.1]{rhodes-vargas-log-kpz} to prove our formulas for the dimension of a set when it is ``zipped up" into an SLE curve.

\begin{proof}[Proof of Theorem~\ref{thm-uniform-dim-time}]
Suppose first that we are in the chordal case.
Let $\gamma = \sqrt\kappa$ (if $\kappa \in (0,4)$) or $\gamma = 4/\sqrt\kappa$ (if $\kappa \in (4,8)$). 
Let $Q$ be as in~\eqref{eqn-lqg-coord} for this choice of $\gamma$.
Let $h_0$ be a free-boundary GFF independent from $\eta$ and let $h := h_0 + \frac{2}{\sqrt\kappa} \log |\cdot|$. Also let $\nu_h$ be the $\gamma$-quantum length measure induced by $h$. 

Recall that for each capacity time $t>0$, the inverse centered Loewner map $f_t^{-1}$ has the same law as the time $t$ centered Loewner map for a reverse SLE$_\kappa$ flow~\cite{schramm-sle}. For $t > 0$, let $h^t :=  h\circ f_t^{-1} + Q\log |(f_t^{-1})'| $.  By~\cite[Theorem~1.2]{shef-zipper}, for each $t > 0$ we have $h^t \eqD h$, modulo additive constant. 

In the case when $\kappa \in (0,4)$, we assume that $\eta$ is parameterized by $\nu_h$-length (which is well-defined by~\cite[Theorem~1.3]{shef-zipper}). In the case when $\kappa  > 4$, we assume that $\eta$ is parameterized by half-plane capacity and for $t>0$ we let $\eta^R_t$ be the be the right outer boundary of the hull generated by $\eta([0,t])$, viewed as a curve from $\eta(t)$ to the rightmost point of $\eta([0,t])\cap \BB R$.  By SLE duality, $\eta_t^R$ for a fixed capacity time $t>0$ is an SLE$_{16/\kappa}$-type curve. More precisely, the construction in~\cite[Sections~1.2.3 and~4.3]{ig4} implies that we can find a chordal space-filling SLE$_{\kappa}$ curve $\eta'$ which traces points in the same order as $\eta$ (in the case $\kappa \geq 8$, we have $\eta' = \eta$ and, when $\kappa \in (4,8)$, $\eta'$ can be obtained by iteratively filling in the ``bubbles" disconnected from $\infty$ by $\eta$).  For $t >0$ let $\sigma_t$ be the time such that $\eta'([0,\sigma_t])$ is the hull generated by $\eta([0,t])$ and for $z\in\BB H$ let $\tau_z$ be the time when $\eta'$ hits $z$. Also let $\acute\eta^R_{\tau_z} $ be the right outer boundary of $\eta'([0,\tau_z])$. The curve $\eta^R_t$ is a.s.\ covered by a countable union of curves of the form $\acute\eta_{\tau_z}^R$ for $z\in \BB Q^2\cap \BB H$. By the construction of space-filling SLE~\cite[Theorem~4.1]{ig4} the law of each curve $\acute\eta_{\tau_z}^R$ stopped at the first time it exits a bounded set at positive distance from $\BB H$ is absolutely continuous with respect to that of a whole-plane SLE$_{16/\kappa}(2-16/\kappa)$ stopped at the same time. In particular, Theorem~\ref{thm-length-kpz} applies to the field $h$ and the curve $\eta_t^R$. We henceforth assume that $\eta_t^R$ is parameterized by quantum length with respect to $h$ (which is well-defined by pushing forward the quantum length measure of $h$ under $f_t^{-1}$).

Now let $Y$ be as in the statement of the lemma and assume without loss of generality that $Y\subset [0,\infty)$. Let $h^t$ be defined as in the beginning of the proof and let
\[
\wh Y^t := \left\{ \nu_{h^t}([0,y]) \,:\, y\in  Y   \right\} .
\] 
By the LQG coordinate change formula and since $\eta$ is parameterized by $\gamma$-quantum length with respect to $h$, on the event $\{f_t^{-1}(Y) \subset \eta\}$ we a.s.\ have
\eqbn
\wh Y^t  =
\begin{cases}
C_t - \eta^{-1} (f_t^{-1} (Y)) ,\quad &\kappa \in (0,4) \\
 (\eta_t^R)^{-1} (f_t^{-1}(Y)) ,\quad &\kappa > 4  
\end{cases}
\eqen
where here $C_t >0$ is the total quantum length of $\eta$ run up to capacity time $t$. 
Therefore, Theorem~\ref{thm-length-kpz} implies that a.s.\ on the event $\{f_t^{-1}(Y) \subset \eta\}$, we have
\eqb \label{eqn-apply-length-kpz-time}
\dim_{\mcl H} f_t^{-1}(Y) = \left(1 + \frac{\gamma^2}{4} \right) \dim_{\mcl H} \wh Y^t     -  \frac{\gamma^2}{8} ( \dim_{\mcl H} \wh Y^t  )^2 .
\eqe
Since changing the additive constant (recall that $h^t \eqD h$, modulo additive constant) amounts to scaling the $\gamma$-quantum length measure by a positive constant,~\cite[Theorem~4.1]{rhodes-vargas-log-kpz} (see also~\cite[Remark~1.2]{ghm-kpz}) implies that a.s.\ 
\eqb \label{eqn-conf-dim-time}
\dim_{\mcl H}  Y  = \Psi_\gamma\left( \dim_{\mcl H} \wh Y^t  \right) ,  
\eqe 
with $\Psi_\gamma$ as in~\eqref{eqn-kpz-function}. Combining~\eqref{eqn-apply-length-kpz-time} and~\eqref{eqn-conf-dim-time} yields the statement of the theorem in the chordal case (here we note that $\Phi_\kappa = \Phi_{16/\kappa}$). The radial version follows from the same argument but with the radial reverse SLE/GFF coupling~\cite[Theorem~5.1]{qle} used in place of~\cite[Theorem~1.2]{shef-zipper}. The whole-plane case is immediate from the radial case since for any $s \in \BB R$, $f_s(\eta|_{[s,\infty)})$ has the law of a radial SLE$_\kappa$.
\end{proof}

We next prove Theorem~\ref{thm-uniform-dim}. The proof is similar to that of Theorem~\ref{thm-uniform-dim-time}, but since we are interested in a complementary connected component of the whole curve (rather than the curve run up to a fixed time) we use a quantum cone or a quantum wedge instead of a free-boundary GFF. There are minor additional complications arising from the possibility that $\eta$ intersects itself. 
One might think that Theorem~\ref{thm-uniform-dim} could be deduced from Theorem~\ref{thm-uniform-dim-time} via a local absolute continuity argument, but we do not see a way to do this since the conformal maps $f_t$ and $g_U$ depend on the whole curve, not just its local behavior.

\begin{proof}[Proof of Theorem~\ref{thm-uniform-dim}]
Suppose we are in the whole-plane case and that $\kappa \in (0,4)$. We note that in this case, $\eta$ intersects itself (and hence has more than one complementary connected component) if and only if $\rho < \kappa/2-2$~\cite[Lemmas 2.4 and 2.6]{ig4}.  
Let $\gamma  = \sqrt\kappa$ and let $Q$ be as in~\eqref{eqn-lqg-coord}. Let $(\BB C , h  , 0 , \infty)$ be a $\big( Q - \frac{1}{2\gamma} (\rho + 2) \big)$-quantum cone (so the weight of the quantum cone, as defined in~\cite{wedges}, is $\rho + 2$), independent from $\eta$, with the circle average embedding. Throughout we assume that $\eta$ is parameterized by $\gamma$-quantum length with respect to $h$. 

Let $\mcl F$ be the $\sigma$-algebra generated by the ordered sequence of quantum boundary lengths of elements of $\mcl U$ with respect to $h$, so that $\mcl F$ is trivial if $\rho \geq \kappa/2-2$, and we order the elements of $\mcl U$ by the time at which $\eta$ finishes tracing their boundary. 
Let $U\in\mcl U$ be chosen in an $\mcl F$-measurable manner. For $a > 0$, let
\eqbn
g_{U,a} := g_U(a \cdot) \quad \op{and} \quad \wt h^a := h \circ g_{U,a}  + Q\log |g_{U,a}'|
\eqen
with $g_U : \BB H\rta U$ as in the statement of the theorem. By~\cite[Theorem~1.5]{wedges}, the conditional law given $\mcl F$ of the quantum surface $(\BB H ,  \wt h^a , 0,\infty)$ is that of a $\big(\frac{\gamma}{2} + Q - \frac{1}{\gamma}(\rho+2)\big)$-quantum wedge (if $\rho \geq \kappa/2-2$) or a single bead of such a wedge with given quantum boundary length (if $\rho \in (-2, \kappa/2-2)$).  

Let $A > 0$ be chosen so that $\wt h^A$ is the circle average embedding of the quantum surface as defined in Definition~\ref{def:circleaverage}.  For each bounded subset of $\BB H$ at positive distance from $\{0\}\cup \bdy \BB D$, the law of the field $\wt h^{A}$ restricted to this set is absolutely continuous with respect to the law of a free-boundary GFF with additive constant chosen such that the semicircle average over $\BB H\cap \bdy \BB D$ is zero, restricted to the same set (this is immediate from the definitions in~\cite[Sections 4.2 and 4.4]{wedges}).   
For $b > 0$, we have $\wt h^{b A} = \wt h^A(b^{-1} \cdot) + Q\log b^{-1}$. 
By the conformal invariance of the GFF and since the law of a GFF plus a deterministic constant is mutually absolutely continuous with respect to the law of a GFF when restricted to bounded sets, it follows that we have the same absolute continuity statement with $\wt h^{b A}$ in place of $\wt h^{ A}$ and $\BB H \cap B_b(0)$ in place of $\BB H \cap \BB D$ for each fixed $b > 0$.  

Now let $Y \subset [0,\infty)$ be as in the statement of the lemma and note that (since we are in the whole-plane case and $\kappa \in (0,4)$) all of $\bdy U$ is traced by $\eta$ and $x_U \not= y_U$, so $U \in \mcl U^- \cap \mcl U^+$. 
For $a > 0$, let
\[
\wh Y^a := \left\{ \nu_{\wt h^{a}}([0,y]) \,:\, y\in  Y   \right\} .
\] 
By the LQG coordinate change formula and since $\eta$ is parameterized by $\gamma$-quantum length with respect to $h$, for each $a > 0$ it is a.s.\ the case that
\eqb \label{eqn-lqg-length-set}
\wh Y^a = \eta^{-1} (g_{U,a}(Y)) - C 
\eqe 
where $C \geq 0$ is a random constant.
In particular $\dim_{\mcl H} \wh Y^a = \dim_{\mcl H} \eta^{-1} (g_{U,a}(Y))$ a.s., so by Theorem~\ref{thm-length-kpz} and the absolute continuity considerations of the preceding paragraph, for each fixed $b>0$, it is a.s.\ the case that
\eqbn
\dim_{\mcl H} g_{U,bA}(Y) = \left(1 + \frac{\gamma^2}{4} \right) \dim_{\mcl H} \wh Y^{bA}    -  \frac{\gamma^2}{8} ( \dim_{\mcl H} \wh Y^{bA} )^2.  
\eqen
This implies that a.s.\
\eqb \label{eqn-apply-length-kpz}
\dim_{\mcl H} g_{U,a}(Y) = \left(1 + \frac{\gamma^2}{4} \right) \dim_{\mcl H} \wh Y^{a}    -  \frac{\gamma^2}{8} ( \dim_{\mcl H} \wh Y^{a} )^2, \quad \text{for Lebesgue a.e.\ $a > 0$. } 
\eqe
By~\cite[Theorem~4.1]{rhodes-vargas-log-kpz} and the above absolute continuity considerations, for each fixed $b > 0$ it is a.s.\ the case that
\eqbn
\dim_{\mcl H}  Y  = \Psi_\gamma\left( \dim_{\mcl H} \wh Y^{bA }   \right) ,  
\eqen
with $\Psi_\gamma$ as in~\eqref{eqn-kpz-function}.
Therefore, a.s.\ 
\eqb \label{eqn-conf-dim0}
\dim_{\mcl H}  Y  = \Psi_\gamma\left( \dim_{\mcl H} \wh Y^{a }   \right)
\quad \text{for Lebesgue a.e.\ $a > 0$. }
\eqe   
Since there are only countably many possible choices of $U\in\mcl U$, combining~\eqref{eqn-apply-length-kpz} and~\eqref{eqn-conf-dim0} yields the statement of the theorem in the whole-plane case when $\kappa \in (0,4)$. 

When $\kappa \in (0,4)$, the statement in the chordal case is proven via the same argument, except that we start with a $2 \gamma + 2Q -\frac{1}{\gamma} (\rho^L + \rho^R-4)$-quantum wedge (equivalently, a quantum wedge of weight $\rho^L + \rho^R -4$) parameterized by a distribution $h$ on $\BB H$ and apply~\cite[Theorem~1.2]{wedges} instead of~\cite[Theorem~1.5]{wedges}. In the case when the parameter of the wedge is in $(Q,Q+\gamma/2)$ (so that it consists of a string of beads) we take $h$ to be the distribution corresponding to a single bead of this wedge. 

When $\kappa \in (4,8)$, we set $\gamma = 4/\sqrt\kappa$ instead of $\gamma = \sqrt\kappa$. The whole-plane (resp.\ chordal) case is treated using a similar argument to the one above except that we start with a $\big( Q - \frac{\gamma}{8} (\rho + 2) \big)$-quantum cone (resp.\ a 
$\big( \frac{4}{\gamma}   - \frac{\gamma }{4}(\rho^L + \rho^R) \big)$-quantum 
wedge, or a single bead of such a wedge in the beaded case) and apply~\cite[Theorem~1.17]{wedges} (resp.~\cite[Theorem~1.16]{wedges}). Here we note that $\Phi_\kappa = \Phi_{16/\kappa}$ and that by SLE duality~\cite{zhan-duality1,zhan-duality2,dubedat-duality,ig1,ig4}, $\bdy U \cap \eta$ is an SLE$_{16/\kappa}$-type curve, so we can apply Theorem~\ref{thm-length-kpz} in essentially the same manner as in the case when $\kappa \in (0,4)$. See the proof of Theorem~\ref{thm-uniform-dim-time} for a similar argument.
\end{proof}

\section{Open questions}
\label{sec-open-problems}

Here we list some open problems which are related to the main results of this paper.

\begin{enumerate}
\item Consider the following heuristic argument for computing the multifractal spectrum of SLE (originally obtained rigorously in~\cite{gms-mf-spec}) using Theorem~\ref{thm-uniform-dim}. Recall that the multifractal spectrum of, say, a whole-plane SLE$_\kappa$ curve $\eta$ is the function $\xi = \xi_\kappa : [-1,1] \rta [0,\infty)$ defined by $\xi(s) = \dim_{\mcl H}\wt\Theta^s$, where 
\eqbn
\wt\Theta^s := \left\{x \in \BB R \,:\, \lim_{\ep\rta 0} \frac{|f'(x  + i\ep)|}{\log \ep^{-1}} = s \right\}
\eqen
for $f : \BB H\rta \BB C\setminus \eta$ a conformal map (it is easy to see that the definition of $\xi$ does not depend on $f$).   
Suppose we are given a deterministic Borel set $Y\subset \BB R$. The points of each set $\wt\Theta^s$ should be evenly spread out over $\BB R$, and we know how much $f$ expands small intervals centered at points of $\wt\Theta^s$. So, it should be possible to derive a formula for $\dim_{\mcl H} f\left( Y \cap \bigcup_{t\in [s-\delta , s+\delta]} \wt\Theta^t  \right)  $ in terms of $\dim_{\mcl H}(Y)$, $\xi$, and $\delta$ for each given $s \in [-1,1]$ and $\delta> 0$. Sending $\delta \rta 0$ and maximizing over $s$ yields a formula for $\dim_{\mcl H}f(Y)$ in terms of $\dim_{\mcl H} Y$ and $\xi$. On the other hand, Theorem~\ref{thm-uniform-dim} gives a formula for $\dim_{\mcl H} f(Y) $ in terms of $\dim_{\mcl H}Y$ and $\kappa$ (for a large number of possible choices of $f$). Comparing these two formulas and letting $\dim_{\mcl H} Y$ vary should allow one to recover $\xi$. Can the above argument be made rigorous? \label{item-multifractal}
\item Does the formula of Theorem~\ref{thm-uniform-dim} hold a.s.\ for every choice of the conformal map $g_U : \BB H\rta U$ simultaneously? The proof of the theorem shows that to obtain an affirmative answer to this question it would be enough to show that for a fixed choice of set $X\subset [0,\infty)$, the KPZ formula of~\cite[Theorem~4.1]{rhodes-vargas-log-kpz} (c.f.~\cite[Remark~1.2]{ghm-kpz}) holds simultaneously a.s.\ for the image of $X$ under every conformal map $\BB H\rta\BB H$ which sends $X$ into $[0,\infty)$. Similarly, does the statement of Theorem~\ref{thm-uniform-dim-time} hold a.s.\ for all times $t$ simultaneously? \label{item-all-map}
\item None of Theorems~\ref{thm-uniform-dim-time},~\ref{thm-uniform-dim},~\ref{thm-length-kpz},~or~\ref{thm-length-kpz'} applies in the case when $\kappa=4$. Can these theorems be extended to the case $\kappa = 4$, possibly using critical ($\gamma=2$) LQG~\cite{shef-deriv-mart,shef-renormalization}? \label{item-kappa=4}
\end{enumerate}

\bibliography{cibib} 
\bibliographystyle{hmralphaabbrv}

\end{document}